\documentclass[12pt]{amsart}
\usepackage{latexsym,amssymb,amsmath,amsthm,amscd,graphicx}
\usepackage{esint} 

\usepackage{color}
\usepackage{url,hyperref}

\numberwithin{equation}{section}
\newtheorem{theorem}{Theorem}[section]
\newtheorem{lemma}[theorem]{Lemma}
\newtheorem{corollary}[theorem]{Corollary}
\newtheorem{proposition}[theorem]{Proposition}

\theoremstyle{definition}

\newtheorem{remark}[theorem]{Remark}
\newtheorem{definition}[theorem]{Definition}
\newtheorem{example}[theorem]{Example}

\theoremstyle{remark}

\title[Generalization for Poincar\'e--Sobolev inequalities with local weights]{Generalization for Poincar\'e--Sobolev inequalities with local weights}
\author{Naoya Hatano}
\date{}

\address{Graduate School of Information Science and Technology, the University of Osaka, 1-5, Yamadaoka, Suita-shi, Osaka 565-0871, Japan}

\email{n.hatano.chuo@gmail.com}

\begin{document}

\maketitle

\begin{abstract}
It is well known the local integral inequality which is called Poincar\'e--Sobolev inequality on each domain cube (or ball).
After that many authors investigated the generalization for this inequality with Muckenhoupt-type weights or global weights which are independent of the domain cube.
In this paper, we investigated similar weighted generalization for this inequality with the local weights which are depending on the domain cube without assuming the Muckenhoupt condition.
Moreover, we considered the weighted generalization for the homogeneous and inhomogeneous-type Sobolev embedding theorems as some applications.
\end{abstract}

{\bf 2020 Classification}
Primary: 42B35,
Secondary: 46E35


\section{Introduction}\label{s:Intro}

For $1\le p,q<\infty$, we assume that
\[
\frac1q
\ge
\frac1p-\frac1n.
\]
Then there exists a constant $C>0$ such that for each cube $Q\subset{\mathbb R}^n$ and Lipchitz function $f$ on $Q$,
\[
\left(
\frac1{|Q|}
\int_Q|f(x)-f_Q|^q\,{\rm d}x
\right)^{\frac1q}
\le C
\ell_Q
\left(
\frac1{|Q|}
\int_Q|\nabla f(x)|\,{\rm d}x
\right)^{\frac1p},
\]
where $\ell_Q$ is the sidelength of $Q$, and $f_Q$ stands for the average $|Q|^{-1}\int_Qf(x)\,{\rm d}x$.
This inequality is called Poincar\'e--Sobolev inequality.
Additionally, the exponent $p^\ast$ which satisfies
\[
\frac1{p^\ast}
=
\frac1p-\frac1n
\]
is called Sobolv conjugate number, and the case $q=p^\ast$ in the above Poincar\'e--Sobolev inequality is sharpest.
As an application of this inequality, it is well known that Harnack's inequality for elliptic partial equations is proved.
When $A(x)$ is a matrix-valued function satisfying the uniform ellipticity condition, for a nonnegative weak solution $u$ to the elliptic partial equation
\[
{\rm div}[A(x)\nabla u(x)]=0
\]
and any cube $Q\subset{\mathbb R}^n$, Harnack's inequality
\[
\sup_Qu
\le C
\inf_Qu
\]
holds.
Moreover, it is known that the Sobolev embedding theorem can be obtained from the Poincar\'e--Sobolev inequality with the simple calculation.
To this end, the purpose of this paper is to give some weighted generalization for the Poincar\'e--Sobolev inequality, and consider the application to some weighted generalization for the Sobolev embedding theorem.

There are some known results for the weighted generalization for the Poincar\'e--Sobolev inequalities.
As some generalization of the form keeping the integration averages, for weights $w,v$, the inequality
\begin{equation}\label{eq:2weight-PS}
\left(
\frac1{v(Q)}
\int_Q|f(x)-f_Q|^qv(x)\,{\rm d}x
\right)^{\frac1q}
\le C
\ell_Q
\left(
\frac1{w(Q)}
\int_Q|\nabla f(x)|w(x)\,{\rm d}x
\right)^{\frac1p}
\end{equation}
is considered by many authors.
At first, the 1-weight cases $w=v$ under the Muckenhoupt $A_p$-condition are given as follows.
Fabes, Kenig and Serapioni \cite{FKS82} gave the case $q=p$ and $w\in A_p({\mathbb R}^n)$.
When the weight $w\in A_r({\mathbb R}^n)$ is given, P\'erez and Rela \cite{PeRe19} introduced some suitable Sobolev conjugate number $p_w^\ast$, and proved the case $q=p_w^\ast$.
Next, the 2-weighted cases are given as follows.
Chanillo and Wheeden \cite{ChWh85} investigated the case $1<p<q<\infty$, $w\in A_p({\mathbb R}^n)$ with some assumption for $w,v$ (see Section \ref{s:comparison}).

In this study, we consider the 2-weighted local norm estimate
\[
\|v(f-f_Q)\|_{L^q(Q)}
\le C_{Q,w,v}
\bigl\|w|\nabla f|\bigr\|_{L^p(Q)}
\]
on each cube $Q\subset{\mathbb R}^n$.
Similar investigation is given by Lorist and Wagenaar \cite{LoWa26-pre}.
They considered the case of assuming the Fujii--Wilson-type local $A_\infty$-weight $v^q$ respect to the cube $Q$, and gave more detailed estimates for $1\le p\le q<\infty$.
Meanwhile, our results are not assuming the $A_\infty$-condition, and include the case $p>q$.
Moreover, we consider the 1-weight case with respect to the case $w=1$, and obtain some results including the critical cases.
Here, P\'erez and Rela \cite{PeRe19} considered the general frameworks which are the cases as replacing the right hand side with the general functional for $Q$.
Although we do not aim the such cases in this paper, we succeeded in refining some results as mentioned in the corollaries in their paper (see Section \ref{s:comparison}).

As a related topic in harmonic analysis, the weighted boundedness of fractional integral operators has been extensively studied.
The 1-weighted $L^p({\mathbb R}^n,w^p)$-$L^q({\mathbb R}^n,w^q)$ boundedness was studied by Muckenhoupt and Wheeden in 1974 \cite{MuWh74}.
Subsequently, the 2-weighted boundedness has been investigated by many authors (see, for example, \cite{CrMo13,Moen09,Perez94,Perez90,SaWh92}).
Our study can be viewed as a local version of these investigations for the fractional integral operator of order 1.

Here, the following is a list of standard notation throughout this paper.
\begin{itemize}
\item If $A\le CB$ for some constant $C>0$, we write $A\lesssim B$ or $B\gtrsim A$.

\item If both $A\lesssim B$ and $A\gtrsim B$ hold, we write $A\approx B$.

\item $L^0({\mathbb R}^n)$ denotes the space of all measurable functions.

\item For $E\subset{\mathbb R}^n$, $L_{\rm loc}^1(E)$ denotes the space of functions integrable on every compact subset of $E$.

\item The symbol ${\mathcal Q}({\mathbb R}^n)$ denotes the collection of all cubes in ${\mathbb R}^n$ whose sides are parallel to the coordinate axes.

\item For $Q\in{\mathcal Q}({\mathbb R}^n)$, $f_Q$ is the integral average of $f$ over $Q$.
Also, the symbol $\fint_Qf(x)\,{\rm d}x$ is defined by the same as $f_Q$.
Namely,
\[
f_Q
=
\fint_Qf(x)\,{\rm d}x
\equiv
\frac1{|Q|}
\int_Qf(x)\,{\rm d}x.
\]

\item The symbol ${\mathcal D}(Q)$ denotes the collection of all cubes obtained by finitely many bisections of $Q$.

\item The symbol $Lip(Q)$ is denotes the family of all Lipschitz continuous functions on $Q\in{\mathcal Q}({\mathbb R}^n)$, and for $E\subset{\mathbb R}^n$, the symbol $Lip_{\rm c}(E)$ denotes the family of all Lipschitz functions with compact support in $E$.

\item For $(j,m)\in{\mathbb Z}\times{\mathbb Z}^n$, the cube $2^{-j}(m+[0,1)^n)$ is called dyadic cube, and these family is denoted by ${\mathcal D}({\mathbb R}^n)$.
\end{itemize}

The remainder of this paper is organized as follows.
In Section \ref{s:weights}, we introduce the families of weights used to state our main results and present some remarks.
In Section \ref{s:Main}, we state our main results.
In Section \ref{s:preliminaries}, we provide some lemmas used to prove the main theorems.
In Section \ref{s:proof}, we prove the main theorems.
In Section \ref{s:comparison}, we compare our results with previous results.
Finally, in Section \ref{s:application}, we give some applications of the main theorems.

\section{Families of weights}\label{s:weights}

In this section, we introduce the families of weights using this study.
Additionally, for $E\subset{\mathbb R}^n$, when
\[
w(x),w^{-1}(x)>0
\quad
\text{a.e. $x\in E$},
\]
$w\in L_{\rm loc}^1(E)$ is called weight on $E$.

At first, as some families of weights using the main results, we define the families of pairs of weights on $Q\in{\mathcal Q}({\mathbb R}^n)$ as follows.

\begin{definition}\label{def:cAQ}
Let $1<p<\infty$, $0<q<\infty$, $1\le r,\tilde{r}<\infty$ and $0<\rho\le\infty$, and let $Q\in{\mathcal Q}({\mathbb R}^n)$.
Define ${\mathcal A}_{p,q,\rho}^{r,\tilde{r}}(Q)$ by the family of all pairs $(w,v)$ of weights on $Q$ with the finite seminorm
\[
[w,v]_{{\mathcal A}_{p,q,\rho}^{r,\tilde{r}}(Q)}
\equiv
\left(
\sum_{R\in{\mathcal D}(Q)}
\left[
\frac{\ell_R}{|R|^{\frac1p-\frac1q}}
\left(\fint_Rw(x)^{-p'r}\,{\rm d}x\right)^{\frac1{p'r}}
\left(
\fint_Rv(x)^{q\tilde{r}}\,{\rm d}x
\right)^{\frac1{q\tilde{r}}}
\right]^\rho
\right)^{\frac1\rho}
\]
with usual modifications made when $\rho=\infty$, and abbreviate
$
{\mathcal A}_{p,q}^{r,\tilde{r}}(Q)
\equiv
{\mathcal A}_{p,q,\infty}^{r,\tilde{r}}(Q)
$.
\end{definition}

Next, we will use the 1-weight classes with respect to the special case $(w,v)=(1,|g|)$ the previous classes.
This case corresponds to an analogue of the Fefferman--Phong inequality
\[
\int_{{\mathbb R}^n}|f(x)|^2V(x)\,{\rm d}x
\lesssim
\|V\|_{{\mathcal M}^{n/2}_q}
\int_{{\mathbb R}^n}|\nabla f(x)|^2\,{\rm d}x,
\]
where ${\mathcal M}^{n/2}_q({\mathbb R}^n)$ is the original Morrey space.
In this study, we use the Morrey spaces restricted to $Q\in{\mathcal Q}({\mathbb R}^n)$ (see \cite{Fefferman83}).

\begin{definition}\label{def:Q-Morrey}
Let $Q\in{\mathcal Q}({\mathbb R}^n)$, and let $0<p,q<\infty$ and $0<r\le\infty$.
The Morrey space ${\mathcal M}^p_{q,r}(Q)$ is defined by family of all $f\in L_{\rm loc}^1(Q)$ with the finite quasi-norm
\[
\|f\|_{{\mathcal M}^p_{q,r}(Q)}
\equiv
\left(
\sum_{R\in{\mathcal D}(Q)}
\left[
|R|^{\frac1p-\frac1q}
\|f\|_{L^q(R)}
\right]^r
\right)^{\frac1r}
\]
with usual modifications made when $r=\infty$, and abbreviate
$
{\mathcal M}^p_q(Q)
\equiv
{\mathcal M}^p_{q,\infty}(Q)
$.
\end{definition}

In this applications, we will use the familty of weights without restriction to $Q\in{\mathcal Q}({\mathbb R}^n)$, and in the definition of (semi)norms of weights, we will introduce the only case defined by the framework of ${\mathcal D}({\mathbb R}^n)$.

\begin{definition}\label{def:cA}
Let $1<p<\infty$, $0<q<\infty$, $1\le r,\tilde{r}<\infty$ and $0<\rho\le\infty$.
Define ${\mathcal A}_{p,q,\rho}^{r,\tilde{r}}({\mathbb R}^n)$ and ${\mathcal A}_{p,q,\rho}^{r,\tilde{r},{\rm loc}}({\mathbb R}^n)$ by the families of all pairs $(w,v)$ of weights on ${\mathbb R}^n$ with the finite seminorms
\[
[w,v]_{{\mathcal A}_{p,q,\rho}^{r,\tilde{r}}}
\equiv
\left(
\sum_{Q\in{\mathcal D}}
\left[
\frac{\ell_Q}{|Q|^{\frac1p-\frac1q}}
\left(\fint_Qw(x)^{-p'r}\,{\rm d}x\right)^{\frac1{p'r}}
\left(
\fint_Qv(x)^{q\tilde{r}}\,{\rm d}x
\right)^{\frac1{q\tilde{r}}}
\right]^\rho
\right)^{\frac1\rho}
\]
and
\[
[w,v]_{{\mathcal A}_{p,q,\rho}^{r,\tilde{r},{\rm loc}}}
\equiv
\left(
\sum_{Q\in{\mathcal D}, \; |Q|\le1}
\left[
\frac{\ell_Q}{|Q|^{\frac1p-\frac1q}}
\left(\fint_Qw(x)^{-p'r}\,{\rm d}x\right)^{\frac1{p'r}}
\left(
\fint_Qv(x)^{q\tilde{r}}\,{\rm d}x
\right)^{\frac1{q\tilde{r}}}
\right]^\rho
\right)^{\frac1\rho}
\]
with usual modifications made when $\rho=\infty$, respectively, and abbreviate
$
{\mathcal A}_{p,q}^{r,\tilde{r}}({\mathbb R}^n)
\equiv
{\mathcal A}_{p,q,\infty}^{r,\tilde{r}}({\mathbb R}^n)
$
and
$
{\mathcal A}_{p,q}^{r,\tilde{r},{\rm loc}}({\mathbb R}^n)
\equiv
{\mathcal A}_{p,q,\infty}^{r,\tilde{r},{\rm loc}}({\mathbb R}^n)
$.
\end{definition}

\begin{definition}
Let $1<p<\infty$, $0<q<\infty$, $1\le r,\tilde{r}<\infty$ and $0<\rho\le\infty$.
Define ${\mathcal A}_{p,q,\rho}^{r,\tilde{r}}({\mathbb R}_+^n)$ and ${\mathcal A}_{p,q,\rho}^{r,\tilde{r},{\rm loc}}({\mathbb R}_+^n)$ by the families of all pairs $(w,v)$ of weights on ${\mathbb R}_+^n$ with the finite seminorms
\[
[w,v]_{{\mathcal A}^{r,\tilde{r}}_{p,q,\rho}({\mathbb R}_+^n)}
\equiv
\left(
\sum_{Q\in{\mathcal D}_+}
\left[
\frac{\ell_Q}{|Q|^{\frac1p-\frac1q}}
\left(\fint_Qw(x)^{-p'r}\,{\rm d}x\right)^{\frac1{p'r}}
\left(
\fint_Qv^{q\tilde{r}}(x)\,{\rm d}x
\right)^{\frac1{q\tilde{r}}}
\right]^\rho
\right)^{\frac1\rho}
\]
and
\[
[w,v]_{{\mathcal A}^{r,\tilde{r},{\rm loc}}_{p,q,\rho}({\mathbb R}_+^n)}
\equiv
\left(
\sum_{Q\in{\mathcal D}_+, \; |Q|\le1}
\left[
\frac{\ell_Q}{|Q|^{\frac1p-\frac1q}}
\left(\fint_Qw(x)^{-p'r}\,{\rm d}x\right)^{\frac1{p'r}}
\left(
\fint_Qv^{q\tilde{r}}(x)\,{\rm d}x
\right)^{\frac1{q\tilde{r}}}
\right]^\rho
\right)^{\frac1\rho}
\]
with usual modifications made when $\rho=\infty$, respectively, where ${\mathcal D}_+({\mathbb R}^n)\equiv\{Q\in{\mathcal D}({\mathbb R}^n)\,:\,Q\subset\overline{{\mathbb R}_+^n}\}$.
\end{definition}

\begin{definition}\label{def:Morrey}
Let $0<p,q<\infty$ and $0<r\le\infty$.
The Morrey spaces ${\mathcal M}^p_{q,r}({\mathbb R}^n)$ and $m^p_{q,r}({\mathbb R}^n)$ are defined by families of all $f\in L_{\rm loc}^1({\mathbb R}^n)$ with the finite quasi-norms
\[
\|f\|_{{\mathcal M}^p_{q,r}}
\equiv
\left(
\sum_{Q\in{\mathcal D}}
\left[
|Q|^{\frac1p-\frac1q}
\|f\|_{L^q(Q)}
\right]^r
\right)^{\frac1r}
\]
and
\[
\|f\|_{m^p_{q,r}}
\equiv
\left(
\sum_{Q\in{\mathcal D}, \; |Q|\le1}
\left[
|Q|^{\frac1p-\frac1q}
\|f\|_{L^q(Q)}
\right]^r
\right)^{\frac1r}
\]
with usual modifications made when $r=\infty$, respectively, and abbreviate
$
{\mathcal M}^p_q({\mathbb R}^n)
\equiv
{\mathcal M}^p_{q,\infty}({\mathbb R}^n)
$
and
$m^p_q({\mathbb R}^n)
\equiv
m^p_{q,\infty}({\mathbb R}^n)
$.
\end{definition}

The family ${\mathcal A}_{p,q}^{r,\tilde{r}}({\mathbb R}^n)$ has been already considered in \cite{Moen09}, and, subsequently, its local-type of this ${\mathcal A}_{p,q}^{r,\tilde{r}}(Q)$ is considered in \cite{LoWa26-pre}.
As far as we know, the family ${\mathcal A}_{p,q}^{r,\tilde{r},{\rm loc}}({\mathbb R}^n)$ has not been used in the literature.
However, our use of this family is based on the approach by Rychkov \cite{Rychkov01}, in which the local Muckenhoupt class is introduced (see Definition \ref{def:Ap} below).
Here, the ${\mathcal A}_{p,q}^{r,\tilde{r}}$-families involve power weights, and hence we obtain the Hardy-type weighted Poincar\'e--Sobolev inequalities.

\begin{example}
According to \cite[Example 113]{SDH20}, when $\alpha>-n$, for any $Q\in{\mathcal Q}({\mathbb R}^n)$,
\[
\int_Q|x|^\alpha\,{\rm d}x
\approx
\max(\ell_Q,|c_Q|)^\alpha
|Q|,
\]
where $c_Q$ is a center of $Q$.
Then we have the following examples of power weights.
\begin{itemize}
\item[{\rm (1)}] For each $Q\in{\mathcal D}({\mathbb R}^n)$, $(|x|^\alpha,|x|^\beta)\in{\mathcal A}_{p,q,\rho}^{r,\tilde{r}}(Q)$ if
\[
\alpha
<
\frac n{p'r},
\quad
\beta
>
-\frac n{q\tilde{r}},
\quad
\begin{cases}
\displaystyle
\frac1q
-
\left(\frac1p-\frac1n\right)
>
\frac{\max(\alpha-\beta,0)}n+\frac1\rho,
& \rho<\infty, \vspace{5pt} \\
\displaystyle
\frac1q
-
\left(\frac1p-\frac1n\right)
\ge
\frac{\max(\alpha-\beta,0)}n,
& \rho=\infty.
\end{cases}
\]

\item[{\rm (2)}] $(|x|^\alpha,|x|^\beta)\in{\mathcal A}_{p,q}^{r,\tilde{r}}({\mathbb R}^n),{\mathcal A}_{p,q}^{r,\tilde{r}}({\mathbb R}_+^n)$ if
\[
-\frac n{q\tilde{r}}
<
\beta
\le
\alpha
<
\frac n{p'r},
\quad
\frac1q
-
\left(\frac1p-\frac1n\right)
=
\frac{\alpha-\beta}n
\]
However, when $\rho<\infty$, $(|x|^\alpha,|x|^\beta)\notin{\mathcal A}_{p,q,\rho}^{r,\tilde{r}}({\mathbb R}^n),{\mathcal A}_{p,q,\rho}^{r,\tilde{r}}({\mathbb R}_+^n)$.

\item[{\rm (3)}] $(|x|^\alpha,|x|^\beta)\in{\mathcal A}_{p,q,\rho}^{r,\tilde{r},{\rm loc}}({\mathbb R}^n),{\mathcal A}_{p,q,\rho}^{r,\tilde{r},{\rm loc}}({\mathbb R}_+^n)$ if
\[
-\frac n{q\tilde{r}}
<
\beta
\le
\alpha
<
\frac n{p'r},
\quad
\begin{cases}
\displaystyle
\frac1q
-
\left(\frac1p-\frac1n\right)
>
\frac{\alpha-\beta}n+\frac1\rho,
& \rho<\infty, \vspace{5pt} \\
\displaystyle
\frac1q
-
\left(\frac1p-\frac1n\right)
\ge
\frac{\alpha-\beta}n,
& \rho=\infty.
\end{cases}
\]

\item[{\rm (4)}] Assume that $\rho<\infty$.
Then
\[
(
\max(|x|^{\alpha_1},|x|^{\alpha_2}),
\min(|x|^{\beta_1},|x|^{\beta_2})
)
\in
{\mathcal A}_{p,q,\rho}^{r,\tilde{r}}({\mathbb R}^n),
{\mathcal A}_{p,q,\rho}^{r,\tilde{r}}({\mathbb R}_+^n)
\]
if
\[
-\frac n{q\tilde{r}}
<
\beta_1\le\beta_2
\le
\alpha_2\le\alpha_1
<
\frac n{p'r},
\quad
\frac{\alpha_1-\beta_1}n
<
\frac1q
-
\left(\frac1p-\frac1n\right)
<
\frac{\alpha_2-\beta_2}n.
\]
\end{itemize}
\end{example}

The space $m^p_q({\mathbb R}^n)$ is called the small Morrey space or the local Morrey space.
The space ${\mathcal M}^p_{q,r}({\mathbb R}^n)$ is called the Bourgain--Morrey space.
This space was first considered by Bourgain \cite{Bourgain91} in the study of the Fourier transform of certain finite measures on the sphere.
The formulation used here was later introduced by Moyua, Vargas and Vega in \cite{MVV99}.
Since then, Bourgain--Morrey spaces have been studied by many authors (see, for example, \cite{Diarra24,DiNa24,HNSH23,HLY23,ZSTYY23}).

\begin{remark}
When $r<\infty$, necessary and sufficient conditions of
\[
{\mathcal M}^p_{q,r}(Q)
\ne
\{0\},
\quad
{\mathcal M}^p_{q,r}({\mathbb R}^n)
\ne
\{0\}
\quad \text{and} \quad
m^p_{q,r}({\mathbb R}^n)
\ne
\{0\}
\]
are
\[
p<r,
\quad
q<p<r
\quad \text{and} \quad
p<r,
\]
respectively.
Thus, here and below, one always assume these conditions for all results using these Bourgain--Morrey spaces.
\end{remark}

Additionally, to compare the known weighted Poincar\'e--Sobolev inequalities and establish the weighted Gagliardo--Nirenberg interpolation inequalities, we introduce the Muckenhoupt $A_p$-classes.

\begin{definition}\label{def:Ap}
Let $1\le p<\infty$ and $Q\in{\mathcal D}(Q)$.
Define $A_p(Q)$, $A_p({\mathbb R}^n)$ and $A_p^{\rm loc}({\mathbb R}^n)$ by the family of all weights $w$ with the finite seminorms
\[
[w]_{A_p(Q)}
\equiv
\begin{cases}
\displaystyle
\sup_{R\in{\mathcal D}(Q)}
\fint_Rw(x)\,{\rm d}x
\;
\|w^{-1}\|_{L^\infty(R)},
& p=1, \\
\displaystyle
\sup_{R\in{\mathcal D}(Q)}
\fint_Rw(x)\,{\rm d}x
\left(
\fint_Rw(x)^{-\frac1{p-1}}\,{\rm d}x
\right)^{p-1},
& p>1,
\end{cases}
\]
\[
[w]_{A_p}
\equiv
\sup_{Q\in{\mathcal Q}}
[w]_{A_p(Q)}
\quad \text{and} \quad
[w]_{A_p^{\rm loc}}
\equiv
\sup_{Q\in{\mathcal Q}, \; |Q|\le1}
[w]_{A_p(Q)},
\]
respectively.
Additionally, set
\[
A_\infty(Q)
\equiv
\bigcup_{p=1}^\infty
A_p(Q),
\quad
A_\infty({\mathbb R}^n)
\equiv
\bigcup_{p=1}^\infty
A_p({\mathbb R}^n),
\quad
A^{\rm loc}_\infty({\mathbb R}^n)
\equiv
\bigcup_{p=1}^\infty
A^{\rm loc}_p({\mathbb R}^n).
\]
Additionally, we introduce as the family of Muckenhoupt weights $A_p({\mathbb R}_+^n)$ by having the $A_p$-norm
\[
[w]_{A_p({\mathbb R}_+^n)}
\equiv
\sup_{Q\in{\mathcal Q}, \; Q\subset{\mathbb R}_+^n}
[w]_{A_p(Q)}.
\]
\end{definition}

The class $A_p({\mathbb R}^n)$ is introduced by Muckenhoupt \cite{Muckenhoupt72}, and the class $A_p^{\rm loc}({\mathbb R}^n)$ is introduced by Rychkov \cite{Rychkov01}.

\section{Main theorems}\label{s:Main}

In this section, we present the main theorems obtained in this paper.

\begin{theorem}\label{main:2weight-1}
Let $1<p,q,r<\infty$.
If $p\le q$, then, for all $Q\in{\mathcal Q}({\mathbb R}^n)$, $(w,v)\in{\mathcal A}^{r,r}_{p,q}(Q)$ and $f\in Lip(Q)$,
\[
\|v(f-f_Q)\|_{L^q(Q)}
\lesssim
[w,v]_{{\mathcal A}^{r,r}_{p,q}(Q)}
\bigl\|w|\nabla f|\bigr\|_{L^p(Q)}.
\]
\end{theorem}

\begin{theorem}\label{main:2weight-2}
Let $1<p,r<\infty$ and $0<q,\rho<\infty$.
If
\[
p>q,
\quad
\frac1q=\frac1p+\frac1\rho,
\]
then, for all $Q\in{\mathcal Q}({\mathbb R}^n)$ and $f\in Lip(Q)$, the following assertions hold{\rm :}
\begin{itemize}
\item[{\rm (1)}] When $q>1$, for all $(w,v)\in{\mathcal A}^{r,r}_{p,q,\rho}(Q)$,
\[
\|v(f-f_Q)\|_{L^q(Q)}
\lesssim
[w,v]_{{\mathcal A}^{r,r}_{p,q,\rho}(Q)}
\bigl\|w|\nabla f|\bigr\|_{L^p(Q)}.
\]

\item[{\rm (2)}] When $q\le1$, for all $(w,v)\in{\mathcal A}^{r,1}_{p,q,\rho}(Q)$,
\[
\|v(f-f_Q)\|_{L^q(Q)}
\lesssim
[w,v]_{{\mathcal A}^{r,1}_{p,q,\rho}(Q)}
\bigl\|w|\nabla f|\bigr\|_{L^p(Q)}.
\]
\end{itemize}
\end{theorem}

Moreover, as corollaries, we obtain the following results.
These results can be viewed as the Fefferman--Phong-type weighted inequalities.

\begin{theorem}\label{main:1weight-1}
Let $1<p<\infty$, $1<q<r<\infty$ and $0<s<\infty$.
If
\[
p\le q,
\quad
\frac1q=\frac1p+\frac1s-\frac1n,
\]
then, for all $Q\in{\mathcal Q}({\mathbb R}^n)$, $f\in Lip(Q)$ and $g\in{\mathcal M}^s_r(Q)$,
\[
\|g(f-f_Q)\|_{L^q(Q)}
\lesssim
\|g\|_{{\mathcal M}^s_r(Q)}
\|\nabla f\|_{L^p(Q)}.
\]
\end{theorem}

\begin{theorem}\label{main:1weight-2}
Let $1<p<\infty$, $0<q,r<\infty$ and $0<s<\rho<\infty$ satisfy
\[
p>q,
\quad
\frac1q=\frac1p+\frac1\rho,
\quad
\frac1q=\frac1p+\frac1s-\frac1n.
\]
Then, for all $Q\in{\mathcal Q}({\mathbb R}^n)$ and $f\in Lip(Q)$, the following assertions hold{\rm :}
\begin{itemize}
\item[{\rm (1)}] If $1<q<r$, then, for all $g\in{\mathcal M}^s_{r,\rho}(Q)$,
\[
\|g(f-f_Q)\|_{L^q(Q)}
\lesssim
\|g\|_{{\mathcal M}^s_{r,\rho}(Q)}
\|\nabla f\|_{L^p(Q)}.
\]

\item[{\rm (2)}] If $q\le1$, then, for all $g\in{\mathcal M}^s_{q,\rho}(Q)$,
\[
\|g(f-f_Q)\|_{L^q(Q)}
\lesssim
\|g\|_{{\mathcal M}^s_{q,\rho}(Q)}
\|\nabla f\|_{L^p(Q)}.
\]
\end{itemize}
\end{theorem}

Moreover, assuming $p<n$, we can include the critical case $q=r$.
In particular, we do not need to distinguish between the cases $q>1$ and $q\le1$.

\begin{theorem}\label{main:1weight-3}
Let $1<p<n$, $0<q<\infty$ and $0<s<\rho\le\infty$ satisfy
\[
\frac1q=\frac1p+\frac1s-\frac1n.
\]
Assume that
\begin{itemize}
\item[{\rm (1)}] $\rho=\infty$ when $p<q$, and

\item[{\rm (2)}] $1/q=1/p+1/\rho$ when $p>q$.
\end{itemize}
Then, for all $Q\in{\mathcal Q}({\mathbb R}^n)$, $f\in Lip(Q)$ and $g\in{\mathcal M}^s_{q,\rho}(Q)$,
\[
\|g(f-f_Q)\|_{L^q(Q)}
\lesssim
\|g\|_{{\mathcal M}^s_{q,\rho}(Q)}
\|\nabla f\|_{L^p(Q)}.
\]
\end{theorem}

However, the double critical case $p=q=r$ does not hold.

\begin{theorem}\label{main:1weight-4}
Let $1<p\le n$.
Then, for all $C>0$, there exist $Q\in{\mathcal Q}({\mathbb R}^n)$, $f\in Lip(Q)$ and $g\in{\mathcal M}^n_p(Q)$ such that
\[
\|g(f-f_Q)\|_{L^p(Q)}
>C
\|g\|_{{\mathcal M}^n_p(Q)}
\|\nabla f\|_{L^p(Q)}.
\]
\end{theorem}

\section{Preliminaries}\label{s:preliminaries}

In this section, we introduce some lemmas to prove our main results.
By the well-known $L^p$-boundedness of the Hardy--Littlewood maximal operator for $p>1$, we can see the boundedness of the local Hardy--Littlewood maximal operator $M^{(\eta)}_Q$ with respect to $Q\in{\mathcal Q}({\mathbb R}^n)$ and $\eta>0$ defined by
\[
M_Qf(x)
\equiv
\sup_{R\in{\mathcal D}(Q)}
\frac{\chi_R(x)}{|R|}
\int_R|f(y)|\,{\rm d}y,
\quad
x\in{\mathbb R}^n,
\]
and $M^{(\eta)}_Qf\equiv(M[|f|^\eta\chi_Q])^{1/\eta}$ for $\eta>0$.

\begin{lemma}\label{lem:local-HL}
Let $Q\in{\mathcal Q}({\mathbb R}^n)$, and let $0<\eta,p\le\infty$.
If $p>\eta$, then, for all $f\in L^p(Q)$,
\[
\|M^{(\eta)}_Qf\|_{L^p(Q)}
\lesssim
\|f\|_{L^p(Q)}.
\]
\end{lemma}

By the mean value theorem, we can easily extract the derivative of $f$ from the deference quotient of $f$.

\begin{lemma}[{\cite[Lemma 11.1.3]{Jost12}}]\label{lem:PWO}
For all $Q\in{\mathcal Q}({\mathbb R}^n)$ and $f\in Lip(Q)$,
\[
|f(x)-f_Q|
\lesssim
\int_Q
\frac{|\nabla f(y)|}{|x-y|^{n-1}}
\,{\rm d}y,
\quad
\text{a.e. $x\in Q$}.
\]
\end{lemma}

A collection ${\mathcal S}\subset{\mathcal Q}({\mathbb R}^n)$ is called a sparse family if for every $Q\in{\mathcal S}$, there exists a measurable subset $E_Q\subset Q$ such that:
\begin{itemize}
\item[(1)] $|Q|\le2|E_Q|$;
\item[(2)] The family $\{E_Q\}_{Q\in{\mathcal S}}$ is pairwise disjoint.
\end{itemize}
Then it is known the following the domination of the oscillation of a function, which is a simplified version of Lemma 5.1 in \cite{LOR17}.

\begin{lemma}\label{lem:osc-sparse}
Let $Q\in{\mathcal Q}({\mathbb R}^n)$.
Then there exists a sparse family ${\mathcal S}\subset{\mathcal D}(Q)$ such that
\[
|f(x)-f_Q|
\lesssim
\sum_{R\in{\mathcal S}}
\chi_R(x)
\fint_R|f(y)-f_R|\,{\rm d}y
\quad
\text{a.e. $x\in Q$}.
\]
\end{lemma}

Although this lemma suffices to prove Theorems \ref{main:2weight-1} and \ref{main:2weight-2}, we additionally apply the Hedberg method to the modification
\[
{\rm OSC}_{{\mathcal S},\alpha}f(x)
\equiv
\sum_{R\in{\mathcal S}}
\frac{\chi_{E_R}(x)}{|R|^\alpha}
\fint_R|f(y)-f_R|\,{\rm d}y
\]
of right hand side of this lemma to prove Theorem \ref{main:1weight-3}.
Therefore, we introduce the weighted Poincar\'e--Sobolev-type inequality for ${\rm OSC}_{{\mathcal S},\alpha}f$.

\begin{lemma}\label{lem:OSCa-sparse}
Let $\alpha\in{\mathbb R}$, $1<p<\infty$, $0<q<\infty$ and $0<s<\rho<\infty$ satisfy
\[
\frac1q=\frac1p+\frac1s-\frac1n+\alpha.
\]
Assume that
\begin{itemize}
\item[{\rm (1)}] $\rho=\infty$ when $p\le q$, and

\item[{\rm (2)}] $1/q=1/p+1/\rho$ when $p>q$.
\end{itemize}
Then, for all $Q\in{\mathcal Q}({\mathbb R}^n)$, $f\in Lip(Q)$ and sparse families ${\mathcal S}\subset{\mathcal D}(Q)$,
\[
\|g\cdot{\rm OSC}_{{\mathcal S},\alpha}f\|_{L^q(Q)}
\lesssim
\|g\|_{{\mathcal M}^s_{q,\rho}(Q)}
\|\nabla f\|_{L^p(Q)}.
\]
\end{lemma}

\begin{proof}
We estimate
\begin{align*}
\int_Q
|g(x){\rm OSC}_{{\mathcal S},\alpha}f(x)|^q
\,{\rm d}x
&=
\sum_{R\in{\mathcal S}}
\int_{E_R}|g(x)|^q\,{\rm d}x
\left[
\frac1{|R|^\alpha}
\fint_R|f(y)-f_R|\,{\rm d}y
\right]^q\\
&\lesssim
\sum_{R\in{\mathcal S}}
\int_{E_R}|g(x)|^q\,{\rm d}x
\left[
\frac{\ell_R}{|R|^\alpha}
\fint_R|\nabla f(z)|\,{\rm d}z
\right]^q\\
&\le
\|g\|_{{\mathcal M}^s_{q,\rho}(Q)}^q
\left(
\sum_{R\in{\mathcal S}}
\left[
|R|^{\frac1p}
\fint_R|\nabla f(z)|\,{\rm d}z
\right]^p
\right)^{\frac qp},
\end{align*}
where the last inequality, we have used the embedding of sequence spaces $\ell^p\hookrightarrow\ell^q$ for the case (1), addtitionally.
Here, by Lemma \ref{lem:local-HL},
\begin{align*}
\sum_{R\in{\mathcal S}}
\left[
|R|^{\frac1p}
\fint_R|\nabla f(z)|\,{\rm d}z
\right]^p
&\lesssim
\sum_{R\in{\mathcal S}}
\bigl\|M_Q[|\nabla f|]\bigr\|_{L^p(E_R)}^p
\le
\bigl\|M_Q[|\nabla f|]\bigr\|_{L^p(Q)}^p\\
&\lesssim
\|\nabla f\|_{L^p(Q)}^p.
\end{align*}
Hence,
\[
\int_Q
|g(x){\rm OSC}_{{\mathcal S},\alpha}f(x)|^q
\,{\rm d}x
\lesssim
\|g\|_{{\mathcal M}^s_{q,\rho}(Q)}^q
\|\nabla f\|_{L^p(Q)}^q.
\]
\end{proof}

\section{Proof of Main theorems}\label{s:proof}

\subsection{Proof of Theorem \ref{main:2weight-1} and Theorem \ref{main:2weight-2} (1)}

Since $q>1$, by duality, we may show the estimate
\[
\int_Q
|h(x)|
\cdot
v(x)|f(x)-f_Q|
\,{\rm d}x
\lesssim
\|h\|_{L^{q'}}
[w,v]_{{\mathcal A}^{r,r}_{p,q,\rho}(Q)}
\bigl\|w|\nabla f|\bigr\|_{L^p(Q)}
\]
for all $h\in L^{q'}({\mathbb R}^n)$.
Here, the case $\rho=\infty$ stands for the conclusion of Theorem \ref{main:2weight-1}.

By Lemma \ref{lem:osc-sparse}, there exists a sparse family ${\mathcal S}\subset{\mathcal D}(Q)$ such that
\[
\int_Q
|h(x)|
\cdot
v(x)|f(x)-f_Q|
\,{\rm d}x
\lesssim
\sum_{R\in{\mathcal S}}
\int_R|h(x)|v(x)\,{\rm d}x
\fint_R|f(y)-f_R|\,{\rm d}y.
\]
On the first term of multiplication in this summation, by H\"older's inequality
\begin{align*}
\int_R|h(x)|v(x)\,{\rm d}x
&\le
\left(
\fint_R|h(x)|^{(qr)'}\,{\rm d}x
\right)^{\frac1{(qr)'}}
\left(
\fint_Rv(x)^{qr}\,{\rm d}x
\right)^{\frac1{qr}}
|R|\\
&\lesssim
\frac{\|M_Q^{((qr)')}h\|_{L^{q'}(E_R)}}{|R|^{\frac1{q'}}}
\left(
\fint_Rv(x)^{qr}\,{\rm d}x
\right)^{\frac1{qr}}
|R|.
\end{align*}
On the second term of multiplication in this summation, by Lemma \ref{lem:PWO} and H\"older's inequality,
\begin{align*}
\fint_R|f(y)-f_R|\,{\rm d}y
&\lesssim
\fint_R\int_R\frac{|\nabla f(z)|}{|y-z|^{n-1}}\,{\rm d}z{\rm d}y
\lesssim
\ell_R
\fint_R|\nabla f(z)|\,{\rm d}z\\
&\le
\ell_R
\left(
\fint_Rw(z)^{-\eta'}\,{\rm d}z
\right)^{\frac1{\eta'}}
\left(
\fint_R\bigl[w(z)|\nabla f(z)\bigr]^\eta\,{\rm d}z
\right)^{\frac1\eta}\\
&\lesssim
\ell_R
\left(
\fint_Rw(z)^{-\eta'}\,{\rm d}z
\right)^{\frac1{\eta'}}
\frac{\bigl\|
M_Q^{(\eta)}\bigl[w|\nabla f|\bigr]
\bigr\|_{L^p(E_R)}}{|R|^{\frac1p}},
\end{align*}
where $\eta$ is a parameter satisfying $\eta'=p'r$.
Then, combining these estimates, H\"older's inequality for sequence norm and Lemma \ref{lem:local-HL}, we have
\begin{align*}
\int_Q
|h(x)|
\cdot
v(x)|f(x)-f_Q|
\,{\rm d}x
&\lesssim
\|M_Q^{((qr)')}h\|_{L^{q'}(Q)}
[w,v]_{{\mathcal A}^{r,r}_{p,q,\rho}(Q)}
\bigl\|
M_Q^{(\eta)}\bigl[w|\nabla f|\bigr]
\bigr\|_{L^p(Q)}\\
&\lesssim
\|h\|_{L^{q'}(Q)}
[w,v]_{{\mathcal A}^{r,r}_{p,q,\rho}(Q)}
\bigl\|w|\nabla f|\bigr\|_{L^p(Q)}.
\end{align*}
Remark on the case $\rho=\infty$ in Theorem \ref{main:2weight-1} that we have used Minkowski's inequality as
\[
\left(
\sum_{R\in{\mathcal S}}
\bigl\|
M_Q^{(\eta)}\bigl[w|\nabla f|\bigr]
\bigr\|_{L^p(E_R)}^q
\right)^{\frac1q}
\le
\bigl\|
M_Q^{(\eta)}\bigl[w|\nabla f|\bigr]
\bigr\|_{L^p(Q)}.
\]
Hence, we finish the proof of Theorem \ref{main:2weight-1} and Theorem \ref{main:2weight-2} (1).

\subsection{Proof of Theorem \ref{main:2weight-2} (2)}

For any sparse fmaily ${\mathcal S}\subset{\mathcal D}(Q)$, setting
\[
{\rm osc}_{\mathcal S}f(x)
=
\sum_{R\in{\mathcal S}}
\chi_R(x)
\fint_R|f(y)-f_R|\,{\rm d}y,
\]
by Lemma \ref{lem:osc-sparse}, we may show the Poincar\'e--Sobolev-type estimate
\[
\|v\cdot{\rm osc}_{\mathcal S}f\|_{L^q(Q)}
\lesssim
[w,v]_{{\mathcal A}^{r,1}_{p,q,\rho}(Q)}
\bigl\|w|\nabla f|\bigr\|_{L^p(Q)}.
\]

Since $q-1\le0$, for each $R\in{\mathcal S}$,
\[
{\rm osc}_{\mathcal S}f(x)^{q-1}
\le
\left(
\fint_R|f(y)-f_R|\,{\rm d}y
\right)^{q-1}
\quad
a.e. \; x\in R.
\]
Then,
\begin{align*}
\int_Q
\bigl[
v(x){\rm osc}_{\mathcal S}f(x)
\bigr]^q
\,{\rm d}x
\le
\sum_{R\in{\mathcal S}}
\int_Rv(x)^q\,{\rm d}x
\left(
\fint_R|f(y)-f_R|\,{\rm d}y
\right)^q.
\end{align*}
Here, by the similar argument in the previous subsection,
\[
\fint_R|f(y)-f_R|\,{\rm d}y
\lesssim
\ell_R
\left(
\fint_Rw(z)^{-\eta'}\,{\rm d}z
\right)^{\frac1{\eta'}}
\frac{\bigl\|
M_Q^{(\eta)}\bigl[w|\nabla f|\bigr]
\bigr\|_{L^p(E_R)}}{|R|^{\frac1p}},
\]
where $\eta$ is a parameter satisfying $\eta'=p'r$.
Consequently, H\"older's inequality for the sequence norm and Lemma \ref{lem:local-HL}, similarly, we obtain the desired estimate.

\subsection{Proof of Theorem \ref{main:1weight-3}}

We set
\[
{\rm osc}_{\mathcal S}f(x)
=
\sum_{R\in{\mathcal S}}
\chi_R(x)
\fint_R|f(y)-f_R|\,{\rm d}y
\]
for any sparse family ${\mathcal S}\subset{\mathcal D}(Q)$ again, and apply Hedberg's method to this function.
Let $\alpha>0$.
For each $R_0\in{\mathcal D}(Q)$ and $x\in R_0$, we estimate
\begin{align*}
{\rm osc}_{\mathcal S}f(x)
&=
\left(
\sum_{R\subset R_0}
+
\sum_{R\supsetneq R_0}
\right)
\chi_R(x)
\fint_R|f(y)-f_R|\,{\rm d}y\\
&\lesssim
|R_0|^\alpha
{\rm OSC}_{{\mathcal S},\alpha}f(x)
+
\sum_{R\supsetneq R_0}
\ell_R
\fint_R|\nabla f(z)|\,{\rm d}z\\
&\lesssim
|R_0|^\alpha
{\rm OSC}_{{\mathcal S},\alpha}f(x)
+
|R_0|^{\frac1n-\frac1p}
\|\nabla f\|_{L^p(Q)},
\end{align*}
where in the second term, we have used Lemma \ref{lem:PWO}.
Since the cube $R_0$ is taken arbitrarily, we obtain
\[
{\rm osc}_{\mathcal S}f(x)
\lesssim
\|
\nabla f
\|_{L^p(Q)}^{1-\theta}
{\rm OSC}_{{\mathcal S},\alpha}
f(x)^\theta,
\]
where we set
\[
\theta
=
1-\frac\alpha{\alpha-\frac1n+\frac1p}.
\]
Since the condition
\[
\frac1q
=
\frac1p+\frac1s-\frac1n
\]
can be rewritten by
\[
\frac1{\theta q}
=
\frac1p+\frac1{\theta s}-\frac1n+\alpha,
\]
we can use Lemma \ref{lem:OSCa-sparse}, and obtain
\begin{align*}
\|g\cdot{\rm osc}_{\mathcal S}f\|_{L^q(Q)}
&\lesssim
\|\nabla f\|_{L^p(Q)}^{1-\theta}
\left\|
|g|^{\frac1\theta}
\cdot
{\rm OSC}_{{\mathcal S},\alpha}f
\right\|_{L^{\theta q}(Q)}^\theta\\
&\lesssim
\|\nabla f\|_{L^p(Q)}^{1-\theta}
\left[
\left\|
|g|^{\frac1\theta}
\right\|_{{\mathcal M}^{\theta s}_{\theta q,\theta\rho}(Q)}
\|\nabla f\|_{L^p(Q)}
\right]^\theta\\
&=
\|g\|_{{\mathcal M}^s_{q,\rho}(Q)}
\|\nabla f\|_{L^p(Q)},
\end{align*}
as desired.

\subsection{Proof of Theorem \ref{main:1weight-4}}

We assume that, for all $Q\in{\mathcal Q}({\mathbb R}^n)$, $f\in Lip(Q)$ and $g\in{\mathcal M}^n_p(Q)$, the inequality
\[
\|g(f-f_Q)\|_{L^p(Q)}
\lesssim
\|g\|_{{\mathcal M}^n_p(Q)}
\|\nabla f\|_{L^p(Q)}
\]
holds.
According to the argument in Section \ref{s:application} below, we obtain
\[
\|g\cdot f\|_{L^p}
\lesssim
\|g\|_{{\mathcal M}^n_p}
\|\nabla f\|_{L^p}.
\]
However, this inequality contradicts the counterexample given in \cite[Proposition 4.1]{SST11}.

\section{Comparison with known results}\label{s:comparison}

In this section, we compare the previous results.
It is well known that if $1\le p_1\le p_2\le\infty$, then $A_{p_1}\subset A_{p_2}$, and that $w\in A_p$ if and only if $w^{1-p'}\in A_{p'}$.
To compare the previous results, we use the following statements for the $A_p$-classes (see \cite[Chapter 8]{Grafakos24} for example).

\begin{proposition}\label{prop:Ap}
Let $1\le p<\infty$ and $w\in A_p$.
Then the following assertions hold.
\begin{itemize}
\item[{\rm (1)}] For all $Q\in{\mathcal Q}({\mathbb R}^n)$ and $E\subset Q$,
\[
\left(
\frac{|E|}{|Q|}
\right)^p
\le
[w]_{A_p}
\frac{w(E)}{w(Q)}.
\]

\item[{\rm (2)}] For all $K>1$, if
\[
1\le r\le
1+\frac{K-1}K
\frac1{c_n[w]_{A_p}}
\]
for some $c_n$ depending only on $n$, then, for all $Q\in{\mathcal Q}({\mathbb R}^n)$,
\[
\left(
\fint_Qw(x)^r\,{\rm d}x
\right)^{\frac1r}
\le K
\fint_Qw(x)\,{\rm d}x.
\]
\end{itemize}
\end{proposition}

By Proposition \ref{prop:Ap} (2), if $w\in A_p$ and $v\in A_\infty$, then, for $r>1$ sufficiently close to $1$,
\[
\left[w^{\frac1p},v^{\frac1q}\right]_{{\mathcal A}^{r,r}_{p,q}(Q)}
\le
2^{\frac1{p'}+\frac1q}
\left[w^{\frac1p},v^{\frac1q}\right]_{{\mathcal A}^{1,1}_{p,q}(Q)}.
\]
Thus, we restrict our attention to estimating the seminorm $[w^{1/p},v^{1/q}]_{{\mathcal A}^{1,1}_{p,q}(Q)}$, and compare it with Theorem \ref{main:2weight-1}.
To the best of our knowledge, there are no results for the case $p>q$ corresponding to Theorem \ref{main:2weight-2} for the weighted Poincar\'e--Sobolev inequality.
Therefore, we omit further disscusion of this case.
Since
\[
\left[w^{\frac1p},w^{\frac1p}\right]_{{\mathcal A}^{1,1}_{p,p}(Q)}
\le
[w]_{A_p}^{\frac1p}
\ell_Q,
\]
we obtain the result in \cite{FKS82} under the assumption $w\in A_p$.

If $w\in A_p$ and
\[
\frac{\ell_R}{\ell_Q}
\left(
\frac{v(R)}{v(Q)}
\right)^{\frac1q}
\lesssim
\left(
\frac{w(R)}{w(Q)}
\right)^{\frac1p}
\]
for any $R\in{\mathcal D}(Q)$, then
\[
\left[w^{\frac1p},v^{\frac1q}\right]_{{\mathcal A}^{1,1}_{p,q}(Q)}
\lesssim
[w]_{A_p}^{\frac1p}
\ell_Q
\frac{v(Q)^{\frac1q}}{w(Q)^{\frac1p}}.
\]
Therefore, assuming $v\in A_\infty({\mathbb R}^n)$, we establish Theorem 1.3 in \cite{ChWh85}.

Define the Sobolev conjugate number $p_w^\ast$ with respect to a weight $w\in A_r$ for $1\le r\le p$ by
\[
\frac1{p_w^\ast}
=
\frac1p-\frac1{nr}.
\]
Then, since 
\begin{align*}
\left[
w^{\frac1p},w^{\frac1{p_w^\ast}}
\right]_{{\mathcal A}^{1,1}_{p,p_w^\ast}(Q)}
\lesssim
[w]_{A_r}^{\frac1{nr}}
[w]_{A_p}^{\frac1p}
\frac{\ell_Q}{w(Q)^{\frac1p-\frac1{p_w^\ast}}},
\end{align*}
we obtain the result in \cite[Corollary 1.15]{PeRe19} for $1<p<n$.

\section{Applications}\label{s:application}

As applications, we present the 2-weighted generalizations for Sobolev embedding theorem and the Gagliardo--Nirenberg interpolation inequailties.

Here and below, we let $\tilde{r}$ such that $\tilde{r}>1$ when we used Theorem \ref{main:2weight-1} and Theorem \ref{main:2weight-2} (1), and $\tilde{r}=1$ when we used Theorem \ref{main:2weight-2} (2).
Fix $f\in Lip_{\rm c}({\mathbb R}^n)$ and $Q\in{\mathcal Q}({\mathbb R}^n)$.
By Theorems \ref{main:2weight-1} and \ref{main:2weight-2},
\begin{align*}
\|v\cdot f\|_{L^q(Q)}
&\le
\|v(f-f_Q)\|_{L^q(Q)}
+
\|v\|_{L^q(Q)}
|f_Q|\\
&\lesssim
[w,v]_{{\mathcal A}^{r,\tilde{r}}_{p,q,\rho}(Q)}
\left(
\bigl\|w|\nabla f|\bigr\|_{L^p(Q)}
+
\frac1{\ell_Q}
\|w\cdot f\|_{L^p(Q)}
\right).
\end{align*}

\subsection{Global extension with 2-weights}

Taking a limit as $\ell_Q\to\infty$ keeping center of $Q$, we have the following weighted Sobolev embedding theorem of the homogeneous-type.

\begin{corollary}\label{cor:Sobolev-hom}
Assume that the parameters $p,q,r,\rho$ satisfy the same assumption in Theorems {\rm \ref{main:2weight-1}} and {\rm \ref{main:2weight-2}}.
If $(w,v)\in{\mathcal A}^{r,\tilde{r}}_{p,q,\rho}({\mathbb R}^n)$, then, for all $f\in Lip_{\rm c}({\mathbb R}^n)$,
\[
\|v\cdot f\|_{L^q({\mathbb R}^n)}
\lesssim
[w,v]_{{\mathcal A}^{r,\tilde{r}}_{p,q,\rho}({\mathbb R}^n)}
\bigl\|w|\nabla f|\bigr\|_{L^p({\mathbb R}^n)}.
\]
\end{corollary}

Moreover, since the norm $\|\cdot\|_{L^p({\mathbb R}^n)}$ can be expressed by amlgam norm as
\[
\|f\|_{L^p({\mathbb R}^n)}
=
\left(
\sum_{m\in{\mathbb Z}^n}
\|f\|_{L^p(m+[0,1)^n)}^p
\right)^{\frac1p},
\]
we also have the inhomogeneous-type 2-weighted Sobolev embedding theorem.

\begin{corollary}\label{cor:Sobolev-inhom}
Assume that the parameters $p,q,r,\rho$ satisfy the same assumption in Theorems {\rm \ref{main:2weight-1}} and {\rm \ref{main:2weight-2}}.
If $(w,v)\in{\mathcal A}^{r,\tilde{r},{\rm loc}}_{p,q,\rho}({\mathbb R}^n)$, then, for all $f\in Lip_{\rm c}({\mathbb R}^n)$,
\[
\|v\cdot f\|_{L^q({\mathbb R}^n)}
\lesssim
[w,v]_{{\mathcal A}^{r,\tilde{r},{\rm loc}}_{p,q,\rho}({\mathbb R}^n)}
\left(
\bigl\|w|\nabla f|\bigr\|_{L^p({\mathbb R}^n)}
+
\|w\cdot f\|_{L^p({\mathbb R}^n)}
\right).
\]
\end{corollary}

Additionally, using the interpolation inequality, we have
\begin{align*}
&\|v(f-f_Q)\|_{L^q(Q)}
\le
\|v(f-f_Q)\|_{L^{p_0}(Q)}^{1-\theta}
\|v(f-f_Q)\|_{L^{q_1}(Q)}^\theta\\
&\quad\lesssim
\left[
\|v\cdot f\|_{L^{p_0}(Q)}
+
\|v\|_{L^{p_0}(Q)}
|f_Q|
\right]^{1-\theta}
\left(
[w,v]_{{\mathcal A}_{p_1,q_1,\rho}^{r,\tilde{r}}(Q)}
\bigl\|w|\nabla f|\bigr\|_{L^{p_1}(Q)}
\right)^\theta\\
&\quad\le
\left[
(1+[v^{p_0}]_{A_{p_0}(Q)})
\|v\cdot f\|_{L^{p_0}(Q)}
\right]^{1-\theta}
\left(
[w,v]_{{\mathcal A}_{p_1,q_1,\rho}^{r,\tilde{r}}(Q)}
\bigl\|w|\nabla f|\bigr\|_{L^{p_1}(Q)}
\right)^\theta.
\end{align*}

Thus, we obtain the following homogeneous and inhomogeneous-type weighted Ggaliardo--Nirenberg interpolation inequalities.

\begin{corollary}\label{cor:wGN}
Let $1\le p_0<\infty$ and $0<q,q_1<\infty$ satisfy
\[
\frac1q=\frac{1-\theta}{p_0}+\frac\theta{q_1}.
\]
Assume that the parameters $p_1,q_1,r,\rho$ satisfy the same assumption in Theorems {\rm \ref{main:2weight-1}} and {\rm \ref{main:2weight-2}}.
Then, for all $f\in Lip_{\rm c}({\mathbb R}^n)$, the following assertions hold.
\begin{itemize}
\item[{\rm (1)}] If $(w,v)\in{\mathcal A}^{r,\tilde{r}}_{p,q,\rho}({\mathbb R}^n)$ and $v^{p_0}\in A_{p_0}({\mathbb R}^n)$, then,
\[
\|v\cdot f\|_{L^q({\mathbb R}^n)}
\lesssim
(1+[v^{p_0}]_{A_{p_0}({\mathbb R}^n)})^{1-\theta}
[w,v]_{{\mathcal A}_{p_1,q_1,\rho}^{r,\tilde{r}}({\mathbb R}^n)}^\theta
\|v\cdot f\|_{L^{p_0}}^{1-\theta}
\bigl\|w|\nabla f|\bigr\|_{L^{p_1}({\mathbb R}^n)}^\theta.
\]

\item[{\rm (2)}] If $(w,v)\in{\mathcal A}^{r,\tilde{r},{\rm loc}}_{p,q,\rho}({\mathbb R}^n)$ and $v^{p_0}\in A^{\rm loc}_{p_0}({\mathbb R}^n)$, then, for all $f\in Lip_{\rm c}({\mathbb R}^n)$,
\begin{align*}
\|v\cdot f\|_{L^q({\mathbb R}^n)}
\lesssim
(1+[v^{p_0}]_{A^{\rm loc}_{p_0}({\mathbb R}^n)})^{1-\theta}
&[w,v]_{{\mathcal A}_{p_1,q_1,\rho}^{r,\tilde{r},{\rm loc}}({\mathbb R}^n)}^\theta
\|v\cdot f\|_{L^{p_0}({\mathbb R}^n)}^{1-\theta}\\
&\times
\left(
\bigl\|w|\nabla f|\bigr\|_{L^{p_1}({\mathbb R}^n)}
+
\|w\cdot f\|_{L^{p_1}({\mathbb R}^n)}
\right)^\theta.
\end{align*}
\end{itemize}
\end{corollary}

Moreover, since
\[
\lim_{j\to\infty}
\bigl([-2^j,0)^{n-1}\times[0,2^j)\bigr)\cup[0,2^j)^n
=
\overline{{\mathbb R}_+^n},
\]
\[
\bigcup_{m\in{\mathbb Z}^{n-1}\times({\mathbb N}\cup\{0\})}
\bigl(m+[0,1)^n\bigr)
=
\overline{{\mathbb R}_+^n},
\]
we can reduce the upper half space ${\mathbb R}_+^n$.
Then we have the following corollaries.

\begin{corollary}
Assume that the parameters $p,q,r,\rho$ satisfy the same assumption in Theorems {\rm \ref{main:2weight-1}} and {\rm \ref{main:2weight-2}}.
Then the following assertions hold{\rm :}
\begin{itemize}
\item[{\rm (1)}] If $(w,v)\in{\mathcal A}^{r,\tilde{r},{\rm loc}}_{p,q,\rho}({\mathbb R}_+^n)$, then, for all $f\in Lip_{\rm c}({\mathbb R}_+^n)$,
\[
\|v\cdot f\|_{L^q({\mathbb R}_+^n)}
\lesssim
[w,v]_{{\mathcal A}^{r,\tilde{r}}_{p,q,\rho}({\mathbb R}_+^n)}
\bigl\|w|\nabla f|\bigr\|_{L^p({\mathbb R}_+^n)}.
\]

\item[{\rm (2)}] If $(w,v)\in{\mathcal A}^{r,\tilde{r},{\rm loc}}_{p,q,\rho}({\mathbb R}_+^n)$, then, for all $f\in Lip_{\rm c}({\mathbb R}_+^n)$,
\[
\|v\cdot f\|_{L^q({\mathbb R}_+^n)}
\lesssim
[w,v]_{{\mathcal A}^{r,\tilde{r},{\rm loc}}_{p,q,\rho}({\mathbb R}_+^n)}
\left(
\bigl\|w|\nabla f|\bigr\|_{L^p({\mathbb R}_+^n)}
+
\|w\cdot f\|_{L^p({\mathbb R}_+^n)}
\right).
\]
\end{itemize}
\end{corollary}

\begin{corollary}
Let $1\le p_0<\infty$ and $0<q,q_1<\infty$ satisfy
\[
\frac1q=\frac{1-\theta}{p_0}+\frac\theta{q_1}.
\]
Assume that the parameters $p_1,q_1,r,\rho$ satisfy the same assumption in Theorems {\rm \ref{main:2weight-1}} and {\rm \ref{main:2weight-2}}.
Then, for all $f\in Lip_{\rm c}({\mathbb R}_+^n)$, the following assertions hold.
\begin{itemize}
\item[{\rm (1)}] If $(w,v)\in{\mathcal A}^{r,\tilde{r}}_{p,q,\rho}({\mathbb R}_+^n)$ and $v^{p_0}\in A_{p_0}({\mathbb R}_+^n)$, then,
\[
\|v\cdot f\|_{L^q({\mathbb R}_+^n)}
\lesssim
(1+[v^{p_0}]_{A_{p_0}({\mathbb R}_+^n)})^{1-\theta}
[w,v]_{{\mathcal A}_{p_1,q_1,\rho}^{r,\tilde{r}}({\mathbb R}_+^n)}^\theta
\|v\cdot f\|_{L^{p_0}}^{1-\theta}
\bigl\|w|\nabla f|\bigr\|_{L^{p_1}({\mathbb R}_+^n)}^\theta.
\]

\item[{\rm (2)}] If $(w,v)\in{\mathcal A}^{r,\tilde{r},{\rm loc}}_{p,q,\rho}({\mathbb R}_+^n)$ and $v^{p_0}\in A^{\rm loc}_{p_0}({\mathbb R}_+^n)$, then, for all $f\in Lip_{\rm c}({\mathbb R}_+^n)$,
\begin{align*}
\|v\cdot f\|_{L^q({\mathbb R}^n)}
\lesssim
(1+[v^{p_0}]_{A^{\rm loc}_{p_0}({\mathbb R}_+^n)})^{1-\theta}
&[w,v]_{{\mathcal A}_{p_1,q_1,\rho}^{r,\tilde{r},{\rm loc}}({\mathbb R}_+^n)}^\theta
\|v\cdot f\|_{L^{p_0}({\mathbb R}_+^n)}^{1-\theta}\\
&\times
\left(
\bigl\|w|\nabla f|\bigr\|_{L^{p_1}({\mathbb R}_+^n)}
+
\|w\cdot f\|_{L^{p_1}({\mathbb R}_+^n)}
\right)^\theta.
\end{align*}
\end{itemize}
\end{corollary}

\subsection{Improvement of unweighted Sobolev embedding theorem}

According to the case $q=1$ of Theorem \ref{main:1weight-2} (2), the Fefferman--Phong-type weighted inequality is given as
\[
\|g\cdot f\|_{L^1}
\lesssim
\|g\|_{{\mathcal M}^{(p^\ast)'}_{1,p'}}
\|\nabla f\|_{L^p},
\quad
\|g\cdot f\|_{L^1}
\lesssim
\|g\|_{m^{(p^\ast)'}_{1,p'}}
\left(
\|\nabla f\|_{L^p}
+
\|f\|_{L^p}
\right),
\]
where the spaces ${\mathcal M}^p_{q,r}({\mathbb R}^n)$ and $m^p_{q,r}({\mathbb R}^n)$ are the original Bourgain--Morrey spaces with the finite norm
\[
\|g\|_{{\mathcal M}^p_{q,r}}
\equiv
\sup_{Q\in{\mathcal D}}
\|g\|_{{\mathcal M}^p_{q,r}(Q)}
=
\left(
\sum_{Q\in{\mathcal D}}
\left[
|Q|^{\frac1p-\frac1q}
\|g\|_{L^q(Q)}
\right]^r
\right)^{\frac1r},
\]
\[
\|g\|_{m^p_{q,r}}
\equiv
\left(
\sum_{m\in{\mathbb Z}^n}
\|g\|_{{\mathcal M}^p_{q,r}(m+[0,1)^n)}^r
\right)^{\frac1r}
=
\left(
\sum_{Q\in{\mathcal D}, \; |Q|\le1}
\left[
|Q|^{\frac1p-\frac1q}
\|g\|_{L^q(Q)}
\right]^r
\right)^{\frac1r},
\]
respectively.
Therefore, using the dual spaces $\dot{W}^{-1,p'}({\mathbb R}^n)$ and $W^{-1,p'}({\mathbb R}^n)$ of the Sobolev spaces
\[
\dot{W}^{1,p}({\mathbb R}^n)
=
\{
f\in L_{\rm loc}^1({\mathbb R}^n)
\,:\,
\|\nabla f\|_{L^p}<\infty
\},
\]
\[
W^{1,p}({\mathbb R}^n)
=
\{
f\in L_{\rm loc}^1({\mathbb R}^n)
\,:\,
\|f\|_{L^p}+\|\nabla f\|_{L^p}<\infty
\},
\]
respectively, we have the following Sobolev embedding theorems.

\begin{corollary}\label{cor:improve-Sobolev}
Let $1<p<n$.
Then the embeddings
\[
{\mathcal M}^p_{1,p^\ast}({\mathbb R}^n)
\hookrightarrow
\dot{W}^{-1,p^\ast}({\mathbb R}^n),
\quad
m^p_{1,p^\ast}({\mathbb R}^n)
\hookrightarrow
W^{-1,p^\ast}({\mathbb R}^n).
\]
\end{corollary}

Here, according to the embedding
\[
L^{p,p^\ast}({\mathbb R}^n)
\hookrightarrow
{\mathcal M}^p_{1,p^\ast}({\mathbb R}^n)
\hookrightarrow
m^p_{1,p^\ast}({\mathbb R}^n)
\]
obtained in \cite[Theorem 3.8]{HNSH23}, this corollary is the improvement of the already known Lorentz--Sobolev embedding theorem
\[
L^{p,p^\ast}({\mathbb R}^n)
\hookrightarrow
\dot{W}^{-1,p^\ast}({\mathbb R}^n).
\]

{\bf Acknowledgements.} 
The author was supported by the Grant-in-Aid for JSPS Fellows (No. 25KJ0222).

\end{document}